\documentclass[10pt]{article}
\usepackage{amsmath,amssymb}
\usepackage{paralist}
\usepackage{graphics} 
\usepackage{epsfig} 
\usepackage{graphicx} 
\usepackage{epstopdf}
\usepackage[colorlinks=true]{hyperref}
\hypersetup{allcolors=blue} 
\usepackage{stmaryrd}

\usepackage{rotating}
\usepackage{adjustbox}
\usepackage{array}      
\usepackage{multicol}        
\usepackage[bottom]{footmisc} 
\usepackage{float}
\usepackage{cancel}
\usepackage[usenames,dvipsnames]{pstricks}
\usepackage{multirow} 
\usepackage{color}
\usepackage{titlesec}
\usepackage{setspace}
\usepackage{colortbl}

\titleformat{\paragraph}
{\normalfont\normalsize\bfseries}{\theparagraph}{1em}{}
\titlespacing*{\paragraph}
{0pt}{3.25ex plus 1ex minus .2ex}{1.5ex plus .2ex}

\newtheorem{theorem}{Theorem}
\newtheorem{definition}{Definition}

\newtheorem{proposition}{Proposition}
\newtheorem{lemma}{Lemma}
\newtheorem{remark}{Remark}

\newcommand{\qed}{\space\hfill\hspace*{\fill} $\vbox{\hrule\hbox{\vrule
height1.3ex\hskip1.3ex\vrule}\hrule}$\hss\vskip\topsep\relax}

\newcommand{\C}{\mathcal{C}}
\newcommand{\R}{\mathbb{R}}

\newcommand{\f}[2]{\frac{\displaystyle{#1}}{\displaystyle{#2}}}
\newcommand{\n}{\newline}

\def\be{\begin{equation}}
\def\ee{\end{equation}}

\def\F{\mathcal{F}}
\def\R{\mathbb{R}}

\def\P{\mathbb{P}}
\def\C{\mathcal{C}}

 \def\A{\mathcal{A}}

\def\f{\frac}

\def\ben{\begin{eqnarray}}
\def\een{\end{eqnarray}}

\def\ezm{e^{-\alpha z^*}}
\def\n{\newline}
\def\v{\varphi}

\def\fibra{\theta_{-t}\omega}

\allowdisplaybreaks

\title{\bf Corrigendum to the chapter ``Some aspects concerning the dynamics of stochastic chemostats"}

\vspace{0.5cm}

\author{Tom\'as Caraballo$^*$, Mar\'{\i}a J. Garrido-Atienza$^*$ and Javier L\'opez-de-la-Cruz\thanks{Partially supported by FEDER and Ministerio de Econom\'ia y Competitividad under grant MTM2015-63723-P, Junta de Andaluc\'{\i}a under the Proyecto de Excelencia P12-FQM-1492 and VI Plan Propio de Investigaci\'on y Transferencia de la Universidad de Sevilla.}\\
{\small\textsl{Dpto. de Ecuaciones Diferenciales y An\'alisis Num\'erico,}}\\
{\small\textsl{Facultad de Matem\'aticas, Universidad de Sevilla}}\\
{\small\textsl{C/ Tarfia s/n, Sevilla 41012, Spain}}\\
{\small\texttt{e-mails: \{caraball,mgarrido,jlopez78\}@us.es}}
\\ \\
Alain Rapaport \\
{\small\textsl{MISTEA (Mathematics, Informatics and Statistics for Environmental and Aggronomic Sciencs),}}\\
{\small\textsl{INRA, Montpellier SupAgro, Univ. Montpellier, 2 place Pierre Viala, 34060 Montpellier cedex 01, France}}\\
{\small\texttt{e-mail: alain.rapaport@inra.fr}}}

\date{}

\begin{document}

\spacing{0.7}

\maketitle

\singlespace
\vspace*{0.5cm}

\noindent {\bf Abstract}\quad {\em In this paper we correct an error made in the paper \cite{CGL}, where a misleading stochastic system was obtained due to a lapse concerning a sign in one of the equations at the beginning of the work such that the results obtained are quite different to the ones developed throughout this paper since the required conditions, and also the results, substantially change. Then, in this work we repair the analysis carried out in \cite{CGL}, where we studied a simple chemostat model influenced by white noise by making use of the theory of random attractors. Even though the changes are minor, we have chosen to provide a new version of the entire paper instead of a list of changes, for sake of readability. We first perform a change of variable using the Ornstein-Uhlenbeck process, transforming our stochastic model into a system of differential equations with random coefficients. After proving that this random system possesses a unique solution for any initial value, we analyze the existence of random attractors. Finally we illustrate our results with some numerical simulations.}

\vspace*{0.5cm}


\vspace*{0.5cm}

\section{Introduction}\label{intro}

Modeling chemostats is a really interesting and important problem with special interest in mathematical biology, since they can be used to study recombinant problems in genetically altered microorganisms (see e.g. \cite{freter,rfreter}), waste water treatment (see e.g. \cite{dans,lariviere}) and play an important role in theoretical ecology (see e.g. \cite{bungay,cunnin,fredrickson,jannash,taylor,veldcamp,w83,w80}). Derivation and analysis of chemostat models are well documented in \cite{sm,swalt,w90} and references therein.\n

Two standard assumptions for simple chemostat models are 1) the availability of the nutrient and its supply rate are fixed and 2) the tendency of the microorganisms to adhere to surfaces is not taken into account. However, these are very strong restrictions as the real world is non-autonomous and stochastic, and this justifies the analysis of stochastic chemostat models.\n

Let us first consider the simplest chemostat model
\begin{eqnarray}
\f{dS}{dt} &=& (S^0-S)D-\f{mSx}{a+S}, \label{I1}
\\[1.3ex]
\f{dx}{dt} &=& x\left(\f{mS}{a+S}-D\right), \label{I2}
\end{eqnarray}
\noindent where $S(t)$ and $x(t)$ denote concentrations of the nutrient and the microbial biomass, respectively; $S^0$ denotes the volumetric dilution rate, $a$ is the half-saturation constant, $D$ is the dilution rate and $m$ is the maximal consumption rate of the nutrient and also the maximal specific growth rate of microorganisms. We notice that all parameters are positive and we use a function Holling type-II, $\mu(S)=mS/(a+S)$, as functional response of the microorganism describing how the nutrient is consumed by the species (see \cite{sree2} for more details and biological explanations about this model).\n

However, we can consider a more realistic model by introducing a white noise in one of the parameters, therefore we replace the dilution rate $D$ by $D+\alpha \dot{W}(t)$, where $W(t)$ is a white noise, i.e., is a Brownian motion, and $\alpha\geq 0$ represents the intensity of noise. Then, system \eqref{I1}-\eqref{I2} is replaced by the following system of stochastic differential equations understood in the It\^o sense
\begin{eqnarray}
dS &=& \left[(S^0-S)D-\f{mSx}{a+S}\right]dt+\alpha(S^0-S)dW(t), \label{I3}
\\[1.3ex]
dx &=& x\left(\f{mS}{a+S}-D\right)dt-\alpha xdW(t). \label{I4}
\end{eqnarray}

System \eqref{I3}-\eqref{I4} has been analyzed in \cite{xu} by using the classic techniques from stochastic analysis and some stability results are provided there. However, as in our opinion there are some unclear points in the analysis carried out there, our aim in this paper is to use an alternative approach to this problem, specifically the theory of random dynamical systems, which will allow us to partially improve the results in \cite{xu}. In addition, we will provide some results which hold almost surely while those from \cite{xu} are said to hold in probability.\n

Firstly, thanks to the well-known conversion between It\^o and Stratonovich sense, we obtain from \eqref{I3}-\eqref{I4} its equivalent Stratonovich formulation which is given by
\begin{eqnarray}
dS&=&\left[(S^0-S)\bar{D}-\frac{mSx}{a+S}\right]dt+\alpha(S^0-S)\circ dW(t),\label{5}
\\[1.3ex]
dx&=&\left[-\bar{D}x+\f{mSx}{a+S}\right]dt-\alpha x\circ dW(t),\label{6}
\end{eqnarray}
\noindent where $\bar{D}:=D+\frac{\alpha^2}{2}$.\n

In Section \ref{preliminaries} we recall some basic results on random dynamical systems. In Section \ref{randomchemostat} we start with the study of equilibria and we prove a result related to the existence and uniqueness of global solution of system \eqref{5}-\eqref{6}, by using the so-called Ornstein-Uhlenbeck  (O-U) process. Then, we define a random dynamical system and prove the existence of a random attractor giving an explicit expression for it. Finally, in Section \ref{nsfc} we show some numerical simulations with different values of the parameters involved in the model and we can see what happens when the amount of noise $\alpha$ increases.


\section{Random dynamical systems}\label{preliminaries}

In this section we present some basic results related to random dynamical systems (RDSs) and random attractors which will be necessary for our analysis. For more detailed information about RDSs and their importance, see \cite{arnold}.\newline

Let $(\mathcal{X},\Vert \cdot \Vert _\mathcal{X})$ be a separable Banach space and let $(\Omega, \F, \P)$ be a probability space where
$\F$ is the $\sigma-$algebra of measurable subsets of $\Omega$ (called
``events") and $\P$ is the probability measure.   To connect the state $\omega$
in the probability space $\Omega$ at time 0 with its state after a time of $t$
elapses, we define a flow $\theta = \{\theta_t \}_{t \in \R}$ on
$\Omega$ with each $\theta_t$ being a mapping $\theta_t: \Omega \to \Omega$ that
satisfies
\begin{itemize} 
\item[(1)]  $\,\,\,\,$ $\theta_0 = \mbox{Id}_\Omega$, 
\item[(2)]  $\,\,\,\,$ $\theta_s \circ \theta_t = \theta_{s+t}$ for all $s, t \in \R$,
\item[(3)]  $\,\,\,\,$ the mapping $(t, \omega) \mapsto \theta_t \omega$ is measurable,
\item[(4)]  $\,\,\,\,$ the probability measure $\P$ is preserved by $\theta_t$, i.e.,  $\theta_t \P = \P$. 
\end{itemize}
This set-up establishes a time-dependent family $\theta$ that tracks the noise,
and $(\Omega, \F, \P, \theta)$ is called a \textit{metric dynamical system} (see \cite{arnold}).
\begin{definition}\label{RDS}
A stochastic process $\{\v(t,\omega )\}_{t\geq 0,\omega \in
\Omega }$ is said to be a \textrm{continuous RDS over $(\Omega ,\mathcal{F},%
\ensuremath{\mathbb{P}},\{\theta _t\}_{t\in \ensuremath{\mathbb{R}}})$ with
state space $\mathcal{X}$} if $\v:[0,+\infty )\times \Omega \times \mathcal{X}\rightarrow \mathcal{X}$ is $(%
\mathcal{B}[0,+\infty )\times \ \mathcal{F}\times \mathcal{B}(\mathcal{X}),\ \mathcal{B%
}(\mathcal{X}))$- measurable, and for each $\omega \in \Omega $,
\begin{itemize}
\item[(i)]   $\,\,\,\,$ the mapping $\v(t,\omega ): \mathcal{X}\rightarrow \mathcal{X}$, $x\mapsto \v(t,\omega
)x$ is continuous for every $t\geq 0$,
\item[(ii)]   $\,\,\,\,$ $\v(0,\omega )$ is the identity operator on $\mathcal{X}$,
\item[(iii)]   $\,\,\,\,$ (cocycle property) $\v(t+s,\omega )=\v(t,\theta _s\omega)\v(s,\omega )$ for all $s,t\geq 0$.
\end{itemize}
\end{definition}
\begin{definition}
Let $(\Omega ,\mathcal{F},\ensuremath{\mathbb{P}})$ be a probability space. A random set $K$ is a measurable subset of $\mathcal{X}\times\Omega$ with respect to the product $\sigma-$algebra $\mathcal{B}(\mathcal{X})\times\mathcal{F}$.\n

The $\omega-$section of a random set $K$ is defined by
$$K(\omega)=\{x\,:\,(x,\omega)\in K\},\quad \omega\in\Omega.$$
In the case that a set $K\subset \mathcal{X}\times \Omega$ has closed or compact $\omega-$sections it is a random set as soon as the mapping $\omega\mapsto d(x,K(\omega))$ is measurable (from $\Omega$ to $[0,\infty)$) for every $x\in \mathcal{X}$, see \cite{crauel2}. Then $K$ will be said to be a closed or a compact, respectively, random set. It will be assumed that closed random sets satisfy $K(\omega)\neq\emptyset$ for all or at least for $\ensuremath{\mathbb{P}}-$almost all $\omega\in\Omega$.
\end{definition}
\begin{remark}
It should be noted that in the literature very often random sets are defined provided that $\omega\mapsto d(x,K(\omega))$ is measurable for every $x\in \mathcal{X}$. Obviously this is satisfied, for instance, when $K(\omega)=N$ for all $\omega$, where $N$ is some non-measurable subset of $\mathcal{X}$, and also when $K=(U\times F)\cup(\overline{U}\times F^c)$ for some open set $U\subset \mathcal{X}$ and $F\notin\mathcal{F}$. In both cases $\omega\mapsto d(x,K(\omega))$ is constant, hence measurable, for every $x\in \mathcal{X}$. However, both cases give $K\subset \mathcal{X}\times\Omega$ which is not an element of the product $\sigma-$algebra $\mathcal{B}(\mathcal{X})\times\mathcal{F}$.
\end{remark}
\begin{definition}
A bounded random set $K(\omega) \subset \mathcal{X}$\ is said to be \textrm{%
tempered with respect to }$\{\theta _t\}_{t\in \ensuremath{\mathbb{R}}}$ if
for a.e. $\omega \in \Omega $, $$\lim_{t\rightarrow \infty } e^{-\beta
t}\sup\limits_{x\in K(\theta _{-t}\omega )}\Vert x\Vert _\mathcal{X}=0, \quad \mbox{for all} \ \beta > 0;$$ a random variable $\omega \mapsto r(\omega )\in \ensuremath{\mathbb{R}}$%
\textrm{\ }is said to be \textrm{tempered with respect to }$\{\theta
_t\}_{t\in \ensuremath{\mathbb{R}}}$ if for a.e. $\omega \in \Omega $, $$%
\lim_{t\rightarrow \infty} e^{-\beta t}\sup\limits_{t\in %
\ensuremath{\mathbb{R}}}|r(\theta _{-t}\omega )|=0, \quad \mbox{for all} \ \beta >0.$$ 
\end{definition}
In what follows we use ${\mathcal{E}}(\mathcal{X})$ to denote the set of all tempered
random sets of $\mathcal{X}$.
\begin{definition}
A  random set $B(\omega )\subset \mathcal{X}$ is called a
\textrm{random absorbing set} in ${\mathcal{E}}(\mathcal{X})$ if for any $E\in
{\mathcal{E}}(\mathcal{X})$ and a.e. $\omega \in \Omega $, there exists $T_E(\omega )>0
$ such that $$\v(t,\theta _{-t}\omega)E(\theta _{-t}\omega )\subset
B(\omega), \quad \forall t\geq T_E(\omega ).$$
\end{definition}
\begin{definition}
Let $\{\v(t,\omega )\}_{t\geq 0,\omega \in \Omega }$ be an RDS
over $(\Omega ,\mathcal{F},\ensuremath{\mathbb{P}},\{\theta _t\}_{t\in %
\ensuremath{\mathbb{R}}})$ with state space $\mathcal{X}$ and let $A(\omega
)(\subset \mathcal{X})$ be a random set.  Then $\A=\{A(\omega)\}_{\omega\in\Omega}$ is called a \textrm{global random ${\mathcal{E}}-$attractor (or pullback $\mathcal{E}-$attractor)
} for $\{\v(t,\omega )\}_{t\geq 0,\omega \in \Omega }$ if
\begin{itemize}
\item[(i)]   $\,\,\,\,$ (compactness) $A(\omega )$ is a compact set of $\mathcal{X}$ for
any $\omega \in \Omega $;
\item[(ii)]   $\,\,\,\,$ (invariance) for any $\omega \in \Omega $ and all $t\geq 0$, it holds $$%
\v(t,\omega)A(\omega)=A(\theta _t\omega );$$
\item[(iii)]   $\,\,\,\,$ (attracting property) for any $E\in {\ \mathcal{E}}(\mathcal{X})$ and
a.e. $\omega \in \Omega $,
 $$
\lim\limits_{t\rightarrow \infty }\mathrm{dist}%
_\mathcal{X}(\v(t,\theta _{-t}\omega )E(\theta _{-t}\omega ),A(\omega ))=0,
$$
where $$\mathrm{dist}_\mathcal{X}(G, H) = \sup_{g \in G} \inf_{h \in H} \|g - h\|_\mathcal{X}$$ is the Hausdorff semi-metric for $G, H \subseteq \mathcal{X}.$
\end{itemize}
\end{definition}
\begin{proposition} \label{attractor} \cite{CLR2006,FS1996} Let $B \in \mathcal{E}(\mathcal{X})$ be a closed absorbing set for the continuous random dynamical 
system $\{\v(t,\omega)\}_{t \geq 0,\omega\in\Omega}$ that satisfies the asymptotic compactness condition for $a. e. \  \omega \in \Omega$, i.e., each sequence 
$x_n \in \v(t_n, \theta_{-t_n}\omega) B (\theta_{-t_n} \omega)$ has a convergent subsequence in $\mathcal{X}$ when $t_n \to \infty$.   Then $\v$ has a unique 
global random attractor $\A=\{A(\omega)\}_{\omega\in\Omega}$ with component subsets 
$$
A(\omega) = \bigcap_{\tau \geq T_B(\omega)} \overline{\bigcup_{t \geq \tau} \v(t, \theta_{-t} \omega) B(\theta_{-t} \omega)}.
$$
If the pullback absorbing set is positively invariant, i.e.,  $\v(t,\omega) B(\omega)$ $\subset$ $B(\theta_{t} \omega)$ for all $t$ $\geq$ $0$, then 
$$
A(\omega) = \bigcap_{t \geq 0}  \overline{\v(t, \theta_{-t} \omega) B(\theta_{-t} \omega)}.
$$
\end{proposition}
\begin{remark}\label{remark1}When the state space $\mathcal{X}$ $=$ $\mathbb{R}^d$ as in this paper, the asymptotic compactness follows trivially.   Note that the random attractor is path-wise attracting 
in the pullback sense, but does not need to be path-wise attracting in the forward sense,  although it is  forward attracting  in probability, due to some possible large deviations, see e.g. \cite{arnold}. 
\end{remark}

The next result ensures when two random dynamical systems are conjugated (see also \cite{caraballoconjugated2,caraballo-book,caraballoconjugated1}).

\begin{lemma}\label{lemmaconjugation}
Let $\varphi_u$ be a random dynamical system on $\mathcal{X}$. Suppose that the mapping $\mathcal{T}:\Omega\times \mathcal{X}\rightarrow \mathcal{X}$ possesses the following properties: for fixed $\omega\in\Omega$, $\mathcal{T}(\omega,\cdot)$ is a homeomorphism on $\mathcal{X}$, and for $x\in \mathcal{X}$, the mappings $\mathcal{T}(\cdot,x)$, $\mathcal{T}^{-1}(\cdot,x)$ are measurable. Then the mapping

$$(t,\omega,x)\rightarrow\varphi_v(t,\omega)x:=\mathcal{T}^{-1}(\theta_t\omega,\varphi_u(t,\omega)\mathcal{T}(\omega,x))$$
is a (conjugated) random dynamical system.
\end{lemma}


\section{Random chemostat}\label{randomchemostat}

In this section we will investigate the stochastic system \eqref{5}-\eqref{6}. To this end, we first transform it into differential equations with random coefficients and without white noise.\n

Let $W$ be a two sided Wiener process. Kolmogorov's theorem ensures that $W$ has a continuous version, that we will denote by $\omega$, whose canonical interpretation is as follows: let $\Omega$ be defined by $$\Omega =\{\omega \in \C(\R, \R): \omega(0) = 0\} = \C_0(\R, \R),$$  $\F$ be the Borel $\sigma-$algebra on $\Omega$ generated by the compact open topology (see \cite{arnold} for details) and $\P$ the corresponding Wiener measure on $\F$.  We consider the Wiener shift flow given by $$\theta_t \omega(\cdot) = \omega(\cdot + t) - \omega(t),\quad t\in \ensuremath{\mathbb{R}},$$ then $(\Omega, \F, \P, \{\theta_t\}_{t \in \R})$ is a metric dynamical system. Now let us introduce the following Ornstein-Uhlenbeck process on $%
(\Omega,\mathcal{F},\ensuremath{\mathbb{P}},\{\theta _t\}_{t\in %
\ensuremath{\mathbb{R}}})$ 
\begin{equation*}
z^*(\theta _t\omega )=-\int\limits_{-\infty }^0e^s\theta _t\omega (s)%
ds,\quad t\in \ensuremath{\mathbb{R}},\quad \omega \in \Omega, \label{delta}
\end{equation*}
which solves the following Langevin equation  (see e.g. \cite{arnold,CL})
\begin{eqnarray*}
dz+zdt =d\omega(t),\quad  t\in\mathbb{R}. \label{OU}
\end{eqnarray*}

\begin{proposition}
\label{property-delta-lm1} (\cite{arnold,CL}) There exists a $
\theta _t$-invariant set $\widetilde{\Omega }\in \mathcal{F}$
 of  $\Omega$  of full $\ensuremath{\mathbb{P}}$ measure such that for $\omega \in
\widetilde{\Omega },$  we have
\begin{itemize}
\item [(i)] $\,\,$ the random variable  $|z^*(\omega )|$  is tempered.

\item [(ii)] $\,\,$ the mapping 
\[
(t,\omega )\rightarrow z^*(\theta _t\omega )=-\int\limits_{-\infty
}^0e^s\omega (t+s)\mathrm{d}s+\omega(t)
\]
is a stationary solution of \eqref{OU}
with continuous trajectories;\n

\item [(iii)] $\,\,$ in addition, for any $\omega \in \tilde \Omega$:
\begin{eqnarray*}
\lim_{t\rightarrow \pm \infty }\frac{|z^*(\theta _t\omega )|}%
t&=& 0;\\
\lim_{t\rightarrow \pm \infty }\frac 1t\int_0^tz^*(\theta _s\omega
)ds&=&0;\\
 \lim_{t\rightarrow \pm \infty }\frac 1t\int_0^t |z^*(\theta _s\omega
)| ds&=& \mathbb{E}[z^*] < \infty.
\end{eqnarray*}
\end{itemize}
\end{proposition}
In what follows we will consider the restriction of the Wiener shift $\theta$ to the set $\tilde \Omega$, and we restrict accordingly the metric dynamical system to this set, that is also a metric dynamical system, see \cite{caraballoconjugated2}. For simplicity, we will still denote the restricted metric dynamical system by the old symbols $(\Omega,\mathcal{F},\ensuremath{\mathbb{P}},\{\theta _t\}_{t\in %
\ensuremath{\mathbb{R}}})$.


\subsection{Stochastic chemostat becomes a random chemostat}

In what follows we use the Ornstein-Uhlenbeck process to transform \eqref{5}-\eqref{6} into a random system. Let us note that analyzing the equilibria we obtain that the only one is the axial equilibrium $(S^0,0)$ and then we define two new variables $\sigma$ and $\kappa$ by

\begin{eqnarray}
\sigma(t) &=&  (S(t)-S^0)e^{\alpha z^*(\theta_t\omega)}, \label{vcsigma}
\\[1.3ex]
\kappa(t) &=& x(t)e^{\alpha z^*(\theta_t\omega)}. \label{vckappa}
\end{eqnarray}

For the sake of simplicity we will write $z^*$ instead of $z^*(\theta_t\omega)$, and $\sigma$ and $\kappa$ instead of $\sigma(t)$ and $\kappa(t)$.\newline

On the one hand, by differentiation, we have
\begin{eqnarray}
\nonumber
d\sigma&=&e^{\alpha z^*(\theta_t\omega)}\cdot dS+\alpha(S-S^0)e^{\alpha z^*(\theta_t\omega)}[-z^*dt+dW]
\\[1.3ex]
\nonumber
&=&-\bar{D}\sigma dt-\frac{m(S^0+\sigma e^{-\alpha z^*(\theta_t\omega)})}{a+S^0+\sigma e^{-\alpha z^*(\theta_t\omega)}}\kappa dt-\alpha z^*\sigma dt.
\end{eqnarray}

On the other hand, we obtain
\begin{eqnarray}
\nonumber
d\kappa&=&e^{\alpha z^*(\theta_t\omega)}\cdot dx+\alpha xe^{\alpha z^*(\theta_t\omega)}[-z^*dt+dW]
\\[1.3ex]
\nonumber
&=&\frac{m(S^0+\sigma e^{-\alpha z^*(\theta_t\omega)})}{a+S^0+e^{-\alpha z^*(\theta_t\omega)}}\kappa dt-\bar{D}\kappa dt-\alpha z^*\kappa dt.
\end{eqnarray}

Thus, we deduce the following random system 
\begin{eqnarray}
\frac{d\sigma}{dt}&=&-(\bar{D}+\alpha z^*)\sigma -\frac{m(S^0+\sigma e^{-\alpha z^*(\theta_t\omega)})}{a+S^0+\sigma e^{-\alpha z^*(\theta_t\omega)}}\kappa,\label{7}
\\[1.3ex]
\frac{d\kappa}{dt}&=&-(\bar{D}+\alpha z^*)\kappa+\frac{m(S^0+\sigma e^{-\alpha z^*(\theta_t\omega)})}{a+S^0+\sigma e^{-\alpha z^*(\theta_t\omega)}}\kappa.\label{8}
\end{eqnarray}


\subsection{Random chemostat generates an RDS}

Next we prove that the random chemostat given by \eqref{7}-\eqref{8} generates an RDS. From now on, we will denote $\mathcal X:=\{(x,y)\in\R^2\,:\, x\in\R,\, y\geq 0\}$, the upper-half plane.

\begin{theorem}\label{theorem1}
For any $\omega\in\Omega$ and any initial value $u_0:=(\sigma_0,\kappa_0)\in \mathcal X$, where $\sigma_0:=\sigma(0)$ and $\kappa_0:=\kappa(0)$, system \eqref{7}-\eqref{8} possesses a unique global solution $u(\cdot;0,\omega,u_0):=(\sigma(\cdot;0,\omega,u_0) , \kappa (\cdot;0,\omega,u_0))\in\C^1([0,+\infty),\mathcal X)$ with $u(0;0,\omega,u_0)=u_0$. Moreover, the solution mapping generates a RDS $\varphi_u:\mathbb R^+\times \Omega\times \mathcal X \rightarrow \mathcal X$ defined as
$$\varphi_u(t,\omega)u_0:=u(t;0,\omega,u_0),\quad \text{for all}\,\, t\in \mathbb R^+, \, u_0\in\mathcal X,\, \omega\in\Omega,$$
\noindent the value at time $t$ of the solution of system \eqref{7}-\eqref{8} with initial value $u_0$ at time zero.
\end{theorem}

\begin{proof}Observe that we can rewrite one of the terms in the previous equations as

\begin{eqnarray*}
\frac{m(S^0+\sigma\ezm)}{a+S^0+\sigma\ezm}\kappa &=& \frac{m(S^0+\sigma\ezm+a-a)}{a+S^0+\sigma\ezm}\kappa=m\kappa-\frac{ma\kappa}{a+S^0+\sigma\ezm}
\end{eqnarray*}
and therefore system \eqref{7}-\eqref{8} turns into

\begin{eqnarray}
\f{d\sigma}{dt} &=& -(\bar{D}+\alpha z^*)\sigma-m\kappa+\frac{ma}{a+S^0+\sigma\ezm}\kappa, \label{a1}
\\[1.3ex]
\f{d\kappa}{dt} &=& -(\bar{D}+\alpha z^*)\kappa+m\kappa-\frac{ma}{a+S^0+\sigma\ezm}\kappa. \label{a2}
\end{eqnarray}

Denoting  $u(\cdot;0,\omega,u_0):=(\sigma(\cdot;0,\omega,u_0) , \kappa (\cdot;0,\omega,u_0))$, system \eqref{a1}-\eqref{a2} can be rewritten as
\begin{eqnarray*}
\f{du}{dt} &=& L(\theta_t\omega)\cdot u+F(u,\theta_t\omega), \label{9}
\end{eqnarray*}
\noindent where
\begin{eqnarray*}
L(\theta_t\omega) &=& \left(\begin{array}{cc}
					   -(\bar{D}+\alpha z^*) & -m \\
					   0 & -(\bar{D}+\alpha z^*)+m
					   \end{array}\right)
\end{eqnarray*}
\noindent and $F:\mathcal X\times[0,+\infty)\longrightarrow\R^2$ is given by
\begin{eqnarray*}
F(\xi,\theta_t\omega) &=& \left(\begin{array}{c}
					   \displaystyle{\frac{ma}{a+S^0+\xi_1e^{-\alpha z^*(\theta_t\omega)}}\xi_2}\\
					   \displaystyle{\frac{-ma}{a+S^0+\xi_1e^{-\alpha z^*(\theta_t\omega)}}\xi_2}
					   \end{array}\right),
\end{eqnarray*}
\noindent where $\xi=(\xi_1,\xi_2)\in\mathcal X$.\n

Since $z^*(\theta_t\omega)$ is continuous, $L$ generates an evolution system on $\R^2$. Moreover, we notice that 
\begin{eqnarray*}
\f{\partial}{\partial \xi_2}\left[\pm\frac{am}{a+S^0+\xi_1\ezm}\xi_2\right] &=& \pm\frac{am}{a+S^0+\xi_1\ezm}
\end{eqnarray*}
and

\begin{eqnarray*}
\f{\partial}{\partial \xi_1}\left[\pm\frac{am}{a+S^0+\xi_1\ezm}\xi_2\right] &=& \mp\frac{am\ezm}{(a+S^0+\xi_1\ezm)^2}\xi_2
\end{eqnarray*}
\noindent thus $F(\cdot,\theta_t\omega)\in\C^1(\mathcal X \times[0,+\infty);\R^2)$ which implies that it is locally Lipschitz with respect to $(\xi_1,\xi_2)\in\mathcal X$. Therefore, thanks to classical results from the theory of ordinary differential equations, system \eqref{7}-\eqref{8} possesses a unique local solution. Now, we are going to prove that the unique local solution of system \eqref{7}-\eqref{8} is in fact a unique global one.\n

By defining $Q(t):=\sigma(t)+\kappa(t)$ it is easy to check that $Q$ satisfies the differential equation
\begin{eqnarray}
\nonumber
\frac{dQ}{dt}&=&-(\bar{D}+\alpha z^*)Q,
\end{eqnarray}
\noindent whose solution is given by the following expression
\begin{eqnarray}
Q(t;0,\omega,Q(0))&=&Q(0)e^{-\bar{D}t-\alpha\int_0^tz^*(\theta_s\omega)ds}.\label{Q}
\end{eqnarray}

The right side of \eqref{Q} always tends to zero when $t$ goes to infinity since $\bar{D}$ is positive, thus $Q$ is clearly bounded. Moreover, since
$$\left.\frac{d\sigma}{dt}\right|_{\sigma=0}=-\frac{mS^0}{a+S^0}\kappa<0$$
\noindent we deduce that, if there exists some $t^*>0$ such that $\sigma(t^*)=0$, we will have $\sigma(t)<0$ for all $t>t^*$. Because of the previous reasoning, we will split our analysis into two different cases.\n

\begin{itemize}
\item {\bf Case 1. $\sigma(t)>0$ for all $t\geq 0$:} in this case, from \eqref{7} we obtain
\begin{eqnarray}
\nonumber
\frac{d\sigma}{dt}&\leq& -(\bar{D}+\alpha z^*)\sigma
\end{eqnarray}
\noindent whose solutions should satisfy
\begin{eqnarray}
\sigma(t;0,\omega,\sigma(0))&\leq&\sigma(0)e^{-\bar{D}t-\alpha\int_0^tz^*(\theta_s\omega)ds} \label{sigma}.
\end{eqnarray}

Since $\bar{D}$ is positive, we deduce that $\sigma$ tends to zero when $t$ goes to infinity, hence $\sigma$ is bounded.

\item {\bf Case 2. there exists $t^*>0$ such that $\sigma(t^*)=0$:} in this case, we already know that $\sigma(t)<0$ for all $t>t^*$ and we claim that the following bound for $\sigma$ holds true
\begin{equation}
\sigma(t;0,\omega,\sigma(0))>-(a+S^0)e^{\alpha z^*(\theta_t\omega)}.\label{sigmabound}
\end{equation}

To prove \eqref{sigmabound}, we suppose that there exists $\bar{t}>t^*>0$ such that $$a+S^0+\sigma(\bar{t})e^{-\alpha z^*(\theta_{\bar{t}}\omega)}=0,$$
\noindent then we can find some $\varepsilon(\omega)>0$ small enough such that $\sigma(t)$ is strictly decreasing and
\begin{equation}
-(\bar{D}+\alpha z^*(\theta_t\omega))-\f{m(S^0+\sigma(t)e^{-\alpha z^*(\theta_t\omega)})}{a+S^0+\sigma(t)e^{-\alpha z^*(\theta_t\omega)}}\kappa(t)>0\label{c}
\end{equation}
\noindent holds for all $t\in[\bar{t}-\varepsilon(\omega),\bar{t})$. Hence, from \eqref{c} we have
$$\f{d\sigma}{dt}(\bar{t}-\varepsilon(\omega))>0,$$
\noindent thus there exists some $\delta(\omega)>0$ small enough such that $\sigma(t)$ is strictly increasing for all $t\in[\bar{t}-\varepsilon(\omega),\bar{t}-\varepsilon(\omega)+\delta(\omega))$, which clearly contradicts  the uniqueness of solution. Hence, \eqref{sigmabound} holds true for all $t\in\mathbb{R}$ and we can also ensure that $\sigma$ is bounded.
\end{itemize}

Since $\sigma+\kappa$ and $\sigma$ are bounded in both cases, $\kappa$ is also bounded. Hence, the unique local solution of system \eqref{7}-\eqref{8} is a unique global one. Moreover, the unique global solution of system \eqref{7}-\eqref{8} remains in $\mathcal{X}$ for every initial value in $\mathcal{X}$ since $\kappa\equiv 0$ solves the same system.\n

Finally, the mapping $\varphi_u: \mathbb R^+\times \Omega \times \mathcal X \rightarrow \mathcal X$ given by
\begin{eqnarray*}
\varphi_u(t,\omega)u_0 &:=& u(t;0,\omega,u_0),\quad \text{for all}\,\, t\geq 0, \,\, u_0\in\mathcal X,\,\, \omega\in\Omega, \label{ds}
\end{eqnarray*}
\noindent defines a RDS generated by the solution of \eqref{7}-\eqref{8}. The proof of this statement follows trivially hence we omit it.\n
\qed
\end{proof}

\subsection{Existence of the pullback random attractor}

Now, we study the existence of the pullback random attractor, describing its internal structure explicitly.

\begin{theorem}\label{t2}
There exists, for any $\varepsilon>0$, a tempered compact random absorbing set $B_\varepsilon(\omega)\in\mathcal{E}(\mathcal X)$ for the RDS $\{\varphi_u(t,\omega)\}_{t\geq 0,\,\omega\in\Omega}$, that is, for any $E(\theta_{-t}\omega)\in\mathcal{E}(\mathcal{X})$ and each $\omega\in\Omega$, there exists $T_E(\omega,\varepsilon)>0$ such that $$\varphi_u(t,\theta_{-t}\omega)E(\theta_{-t}\omega)\subseteq B_\varepsilon(\omega),\quad\quad\text{for all}\,\, t\geq T_E(\omega,\varepsilon).$$
\end{theorem}

\begin{proof}Thanks to \eqref{Q}, we have
\begin{eqnarray}
\nonumber
Q(t;0,\fibra,Q(0))&=&Q(0)e^{-\bar{D}t-\alpha\int_{-t}^0z^*(\theta_s\omega)ds}\stackrel{\quad t\rightarrow+\infty\quad}{\longrightarrow}0.
\end{eqnarray}

Then, for any $\varepsilon>0$ and $u_0\in E(\fibra)$ there exists $T_E(\omega,\varepsilon)>0$ such that, for all $t\geq T_E(\omega,\varepsilon)$, we obtain
$$-\varepsilon\leq Q(t;0,\fibra,u_0)\leq \varepsilon.$$

If we assume that $\sigma(t)\geq 0$ for all $t\geq 0$, which corresponds to {\bf Case 1} in the proof of Theorem \ref{theorem1}, since $\kappa(t)\geq 0$ for all $t\geq 0$, we have that
$$B^1_\varepsilon(\omega):=\left\{(\sigma,\kappa)\in\mathcal{X}\,\,:\,\, \sigma\geq 0,\, \sigma+\kappa\leq\varepsilon\right\}$$
\noindent is a tempered compact random absorbing set in $\mathcal{X}$.\n

In the other case, i.e., if there exists some $t^*>0$ such that $\sigma(t^*)=0$, which corresponds to {\bf Case 2} in the proof of Theorem \ref{theorem1}, we proved that 
\begin{equation*}
\sigma(t;0,\fibra,u_0)>-(a+S^0)e^{\alpha z^*(\omega)}.
\end{equation*}

Hence, we obtain that
$$B^2_\varepsilon(\omega):=\left\{(\sigma,\kappa)\in\mathcal{X}\,\,:\,\, -\varepsilon-(a+S^0)e^{\alpha z^*(\omega)}\leq\sigma\leq 0,\, -\varepsilon\leq\sigma+\kappa\leq\varepsilon\right\}$$
\noindent is a tempered compact random absorbing set in $\mathcal{X}$.\n

In conclusion, defining
\begin{equation*}
B_\varepsilon(\omega)=B^1_\varepsilon(\omega)\cup B^2_\varepsilon(\omega)=\left\{(\sigma,\kappa)\in\mathcal{X}\,\,:\,\,-\varepsilon\leq\sigma+\kappa\leq\varepsilon,\, \sigma\geq-(a+S^0)e^{\alpha z^*(\omega)}-\varepsilon\right\},
\end{equation*}
\noindent we obtain (see Figure \ref{figure1}) that $B_\varepsilon(\omega)$ is a tempered compact random absorbing set in $\mathcal{X}$ for every $\varepsilon>0$.

\begin{multicols}{2}
\begin{figure}[H]
\begin{center}
\psscalebox{0.75 0.75} 
{
\begin{pspicture}(0,-4.430296)(10.480869,4.430296)
\definecolor{colour0}{rgb}{0.0,0.6,1.0}
\definecolor{colour1}{rgb}{1.0,0.2,0.2}
\definecolor{colour2}{RGB}{0,102,0}
\pspolygon[linecolor=white, linewidth=0.04, fillstyle=solid,fillcolor=blue](2.3836365,-1.6571428)(5.637391,-1.6650479)(1.1915416,2.8242803)(1.2018182,-0.16623364)(1.2018182,-0.45714274)
\rput[l](10.2808695,-1.5828345){\large $\sigma$}
\rput[b](3.9796839,4.2242804){\large $\kappa$}
\psline[linecolor=colour1, linewidth=0.04, linestyle=dashed, dash=0.17638889cm 0.10583334cm](1.2018182,3.4883118)(1.2018182,-3.238961)
\rput[t](1.18917,-3.40655){\large $-(a+S^0)e^{\alpha z^*(\omega)}-\varepsilon$}
\pspolygon[linecolor=white, linewidth=0.04, fillstyle=solid,fillcolor=colour2](3.9931226,0.004517342)(3.992332,-1.6753246)(5.637391,-1.666629)
\psline[linecolor=black, linewidth=0.04, dotsize=0.07055555555555555cm 2.0,arrowsize=0.05291666666666667cm 2.0,arrowlength=1.4,arrowinset=0.0]{*->}(0.0,-1.6571428)(10.14,-1.6571428)
\psline[linecolor=black, linewidth=0.04, dotsize=0.07055555555555555cm 2.0,arrowsize=0.05291666666666667cm 2.0,arrowlength=1.4,arrowinset=0.0]{*->}(3.9836364,-4.4302654)(4.001818,4.1246753)
\rput[t](5.532253,-1.85){\large $\varepsilon$}
\rput[t](2.4160473,-1.76){\large $-\varepsilon$}
\rput[t](2.5,0.1){\large \bf \textcolor{white}{$B^2_\varepsilon(\omega)$}}
\rput[t](4.5,-1){\bf \textcolor{white}{$B^1_\varepsilon(\omega)$}}
\end{pspicture}
}
\end{center}
\caption{Absorbing set \bf $B_\varepsilon(\omega):=\textcolor[RGB]{0,102,0}{B^2_\varepsilon(\omega)}\cup\textcolor{blue}{B^1_\varepsilon(\omega)}$}
\label{figure1}
\end{figure}
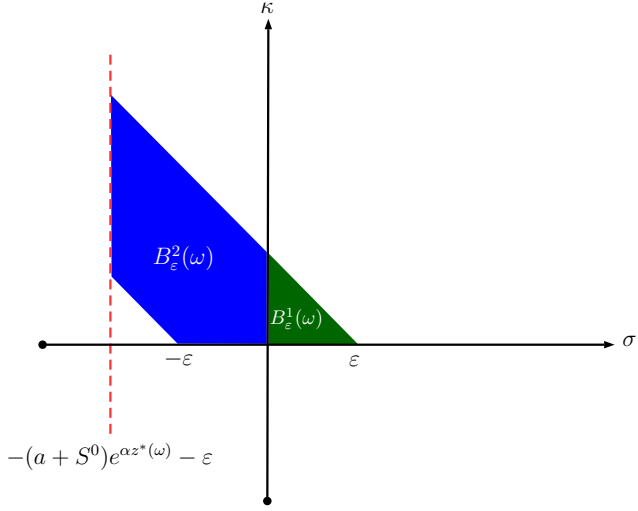

\columnbreak

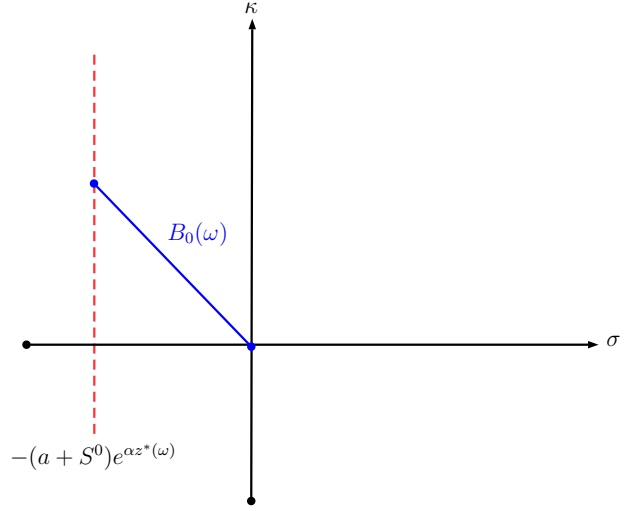
\begin{figure}[H]
\begin{center}
\psscalebox{0.75 0.75} 
{
\begin{pspicture}(0,-4.430296)(10.480869,4.430296)
\definecolor{colour1}{rgb}{1.0,0.2,0.2}
\rput[l](10.2808695,-1.5828345){\large $\sigma$}
\rput[b](3.9796839,4.2242804){\large $\kappa$}
\psline[linecolor=colour1, linewidth=0.04, linestyle=dashed, dash=0.17638889cm 0.10583334cm](1.2018182,3.4883118)(1.2018182,-3.238961)
\rput[t](1.18917,-3.40655){\large $-(a+S^0)e^{\alpha z^*(\omega)}$}
\psline[linecolor=black, linewidth=0.04, dotsize=0.07055555555555555cm 2.0,arrowsize=0.05291666666666667cm 2.0,arrowlength=1.4,arrowinset=0.0]{*->}(0.0,-1.6571428)(10.14,-1.6571428)
\psline[linecolor=black, linewidth=0.04, dotsize=0.07055555555555555cm 2.0,arrowsize=0.05291666666666667cm 2.0,arrowlength=1.4,arrowinset=0.0]{*->}(3.9836364,-4.4302654)(4.001818,4.1246753)
\psline[linecolor=blue, linewidth=0.04, dotsize=0.07055555555555555cm 2.0]{*-*}(3.984706,-1.6924368)(1.1964706,1.2016808)
\rput[bl](2.5,0.1){\large \bf \textcolor{blue}{$B_0(\omega)$}}
\end{pspicture}
}
\end{center}
\caption{Absorbing set \textcolor{blue}{\bf $B_0(\omega)$}}
\label{figure2}
\end{figure}
\end{multicols}

\qed
\end{proof}

Then, thanks to Proposition \ref{attractor}, it follows directly that system \eqref{7}-\eqref{8} possesses a unique pullback random attractor given by 
$$\mathcal{A}(\omega)\subseteq B_\varepsilon(\omega),\quad\quad \text{for all}\,\, \varepsilon>0,$$ 
\noindent thus 
$$\mathcal{A}(\omega)\subseteq B_0(\omega),$$ 
\noindent where
$$B_0(\omega):=\left\{(\sigma,\kappa)\in\mathcal{X}\,\,:\,\, \sigma+\kappa=0,\, \sigma\geq-(a+S^0)e^{\alpha z^*(\omega)}\right\}$$
\noindent is a tempered compact random absorbing set (see Figure \ref{figure2}) in $\mathcal{X}$.\n

The following result provides information about the internal structure of the unique pullback random attractor.

\begin{proposition}\label{pe}
The unique pullback random attractor of system \eqref{7}-\eqref{8} consists of a singleton component given by $\mathcal{A}(\omega)=\{(0,0)\}$ as long as
\begin{equation}
\bar{D}>\mu(S^0)\label{ce}
\end{equation}
\noindent holds true.
\end{proposition}

\begin{proof}We would like to note that the result in this proposition follows trivially if $\sigma$ remains always positive ({\bf Case 1} in the proof of Theorem \ref{theorem1}) since in that case both $\sigma$ and $\kappa$ are positive and $\sigma+\kappa$ tends to zero when $t$ goes to infinity, thus the pullback random attractor is directly given by $\mathcal{A}(\omega)=\{(0,0)\}$.\n

Due to the previous reason, we will only present the proof in case of there exists some $t^*>0$ such that $\sigma(t^*)=0$ which implies that $\sigma(t)<0$ for all $t>t^*$ whence $S(t)<S^0$ for all $t>t^*$ then $\mu(S)\leq\mu(S^0)$ for all $t>t^*$ since $\mu(s)=ms/(a+s)$ is an increasing function. Hence, from \eqref{8} we have
\begin{eqnarray}
\nonumber
\frac{d\kappa}{dt}&\leq&-(\bar{D}+\alpha z^*)\kappa+\frac{mS^0}{a+S^0}\kappa,
\end{eqnarray}
\noindent which allows us to state the following inequality
\begin{eqnarray}
\nonumber
\kappa(t;t^*,\fibra,\kappa(t^*))&\leq&\kappa(t^*)e^{-\left(\bar{D}-\frac{mS^0}{a+S^0}\right)(t-t^*)-\alpha\int_{-t}^{t^*}z^*(\theta_s\omega)ds},
\end{eqnarray}
\noindent where the right side tends to zero when $t$ goes to infinity as long as \eqref{ce} is fulfilled, therefore the unique pullback random attractor is given by $\mathcal{A}(\omega)=\{(0,0)\}.$\n
\qed
\end{proof}

\subsection{Existence of the pullback random attractor for the stochastic chemostat}

We have proved that the system \eqref{7}-\eqref{8} has a unique global solution $u(t;0,\omega,u_0)$ which remains in $\mathcal X$ for all $u_0\in\mathcal X$ and generates the RDS $\{\varphi_u(t,\omega)\}_{t\geq 0,\omega\in\Omega}$.\n

Now, we define a mapping 
$$\mathcal{T}: \Omega\times \mathcal X \longrightarrow \mathcal X$$
as follows
$$\mathcal{T}(\omega,\zeta)=\mathcal{T}(\omega,(\zeta_1,\zeta_2))=  \left(\begin{array}{c}
    T_1(\omega,\zeta_1) \\ 
    T_2(\omega,\zeta_2)\\ 
  \end{array}\right)=\left(\begin{array}{c}
    (\zeta_1-S^0)e^{\alpha z^*(\omega)} \\ 
    \zeta_2 e^{\alpha z^*(\omega)}\\ 
  \end{array}\right)$$
whose inverse is given by
$$\mathcal{T}^{-1}(\omega,\zeta)= \left(\begin{array}{c}
    S^0+\zeta_1 e^{-\alpha z^*(\omega)} \\ 
    \zeta_2 e^{-\alpha z^*(\omega)}\\ 
  \end{array}\right).$$

We know that $v(t)=(S(t),x(t))$ and $u(t)=(\sigma(t),\kappa(t))$ are related by (\ref{vcsigma})-(\ref{vckappa}). Since $T$ is a homeomorphism, thanks to Lemma \ref{lemmaconjugation} we obtain a conjugated RDS given by
\begin{eqnarray}
\nonumber \varphi_v(t,\omega)v_0 &:=& \mathcal{T}^{-1}(\theta_t \omega,\varphi_u(t,\omega)\mathcal{T}(\omega,v_0))
\\[1.3ex]
\nonumber &=& \mathcal{T}^{-1}\left(\theta_t\omega,\varphi_u (t,\omega)\left(\begin{array}{c}
    (S(0)-S^0)e^{\alpha z^*(\omega)} \\ 
    x(0) e^{\alpha z^*(\omega)}\\ 
  \end{array}\right)\right)
\\[1.3ex]
\nonumber &=& \mathcal{T}^{-1}(\theta_t\omega,\varphi_u(t,\omega)u_0)
\\[1.3ex]
\nonumber &=& \mathcal{T}^{-1}(\theta_t\omega,u(t;0,\omega,u_0))
\\[1.3ex]
\nonumber &=& \left(\begin{array}{c}
    S^0+\sigma(t)e^{-\alpha z^*(\theta_t \omega)} \\ 
    \kappa(t) e^{-\alpha z^*(\theta_t \omega)}\\ 
  \end{array}\right)
\\[1.3ex]
\nonumber &=& v(t;0,\omega,v_0)
\end{eqnarray}
which means that $\{\varphi_v(t,\omega)\}_{t\geq 0,\omega\in\Omega}$ is an RDS for our original stochastic system \eqref{5}-\eqref{6} whose unique pullback random attractor satisfies that $\widehat{\mathcal{A}}(\omega) \subseteq \widehat{B}_0(\omega)$, where
\begin{equation}
\widehat{B}_0(\omega):=\left\{(S,x)\in\mathcal{X}\,\,:\,\, S+x=S^0,\,\, S\geq-a\right\}.\label{absb}
\end{equation}

In addition, under \eqref{ce}, the unique pullback random attractor for \eqref{5}-\eqref{6} reduces to a singleton subset $\widehat{\mathcal{A}}(\omega)=\{(S^0,0)\}$, which means that the microorganisms become extinct.\n

We remark that it is not possible to provide conditions which ensure the persistence of the microbial biomass even though our numerical simulations will show that we can get it for many different values of the parameters involved in the system, as we will present in Section \ref{nsfc}.

\section{Numerical simulations and final comments}\label{nsfc}

To confirm the results provided through this paper, in this section we will show some numerical simulations concerning the original stochastic chemostat model given by system \eqref{5}-\eqref{6}. To this end, we will make use of the Euler-Maruyama method (see e.g. \cite{maru} for more details) which consists of considering the following numerical scheme:

\begin{eqnarray*}
S_j &=& S_{j-1}+f(x_{j-1},S_{j-1})\Delta t+g(x_{j-1},S_{j-1})\cdot(W(\tau_j)-W(\tau_{j-1})),
\\[1.3ex]
x_j &=& x_{j-1}+\widetilde{f}(x_{j-1},S_{j-1})\Delta t+\widetilde{g}(x_{j-1},S_{j-1})\cdot(W(\tau_j)-W(\tau_{j-1})),
\end{eqnarray*}
\noindent where $f$, $g$, $\widetilde{f}$ and $\widetilde{g}$ are functions defined as follows

\begin{eqnarray*}
f(x_{j-1},S_{j-1}) &=& \left[(S^0-S_{j-1})D-\f{mS_{j-1}x_{j-1}}{a+S_{j-1}}\right],
\\[1.3ex]
g(x_{j-1},S_{j-1}) &=& \alpha(S^0-S_{j-1}),
\\[1.3ex]
\widetilde{f}(x_{j-1},S_{j-1}) &=& x_{j-1}\left(\f{mS_{j-1}}{a+S_{j-1}}-D\right),
\\[1.3ex]
\widetilde{g}(x_{j-1},S_{j-1}) &=& \alpha x_{j-1},
\\[1.3ex]
\end{eqnarray*}
\noindent and we remark that

$$W(\tau_j)-W(\tau_{j-1})=\sum_{k=jR-R+1}^{jR}dW_k,$$
where $R$ is a nonnegative integer number and $dW_k$ are $\mathcal{N}(0,1)-$distributed independent random variables which can be generated numerically by pseudorandom number generators.\n

From now on, we will display the phase plane $(S,x)$ of the dynamics of our chemostat model, where the blue dashed lines represent the solutions of the deterministic (i.e., with $\alpha=0$) system \eqref{I1}-\eqref{I2} and the other ones are different realizations of the stochastic chemostat model \eqref{5}-\eqref{6}. In addition, we will set $S^0=1$, $a=0.6$, $m=3$ and we will consider $(S(0),x(0))=(2.5,5)$ as initial pair. We will also present different cases where the value of the dilution rate and the amount of noise change in order to obtain different situations in which the condition \eqref{ce} is (or is not) fulfilled.\n

On the one hand, in Figure \ref{sim1} we take $D=3$ and we choose $\alpha=0.1$ (left) and $\alpha=0.5$ (right). In both cases, it is easy to check that $\bar{D}=1.5050$ (left), $\bar{D}=1.6250$ (right) and $\mu(S^0)=1.8750$ thus, thanks to Proposition \ref{pe}, we know that the microorganisms become extinct, as we show in the simulations.

\begin{figure}[H]
\begin{center}
\noindent\includegraphics[scale=0.28]{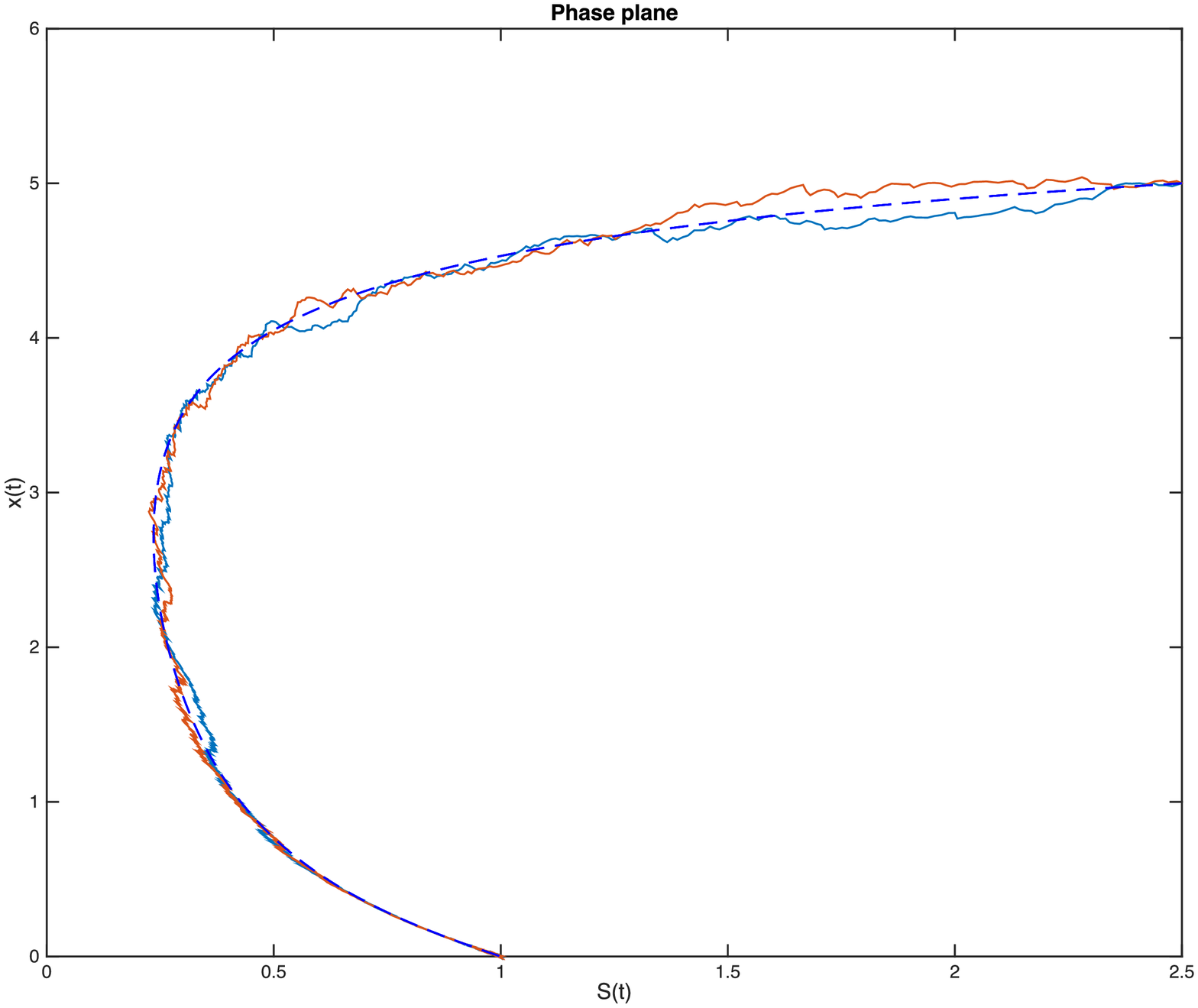}\hspace{-0.5cm}\includegraphics[scale=0.28]{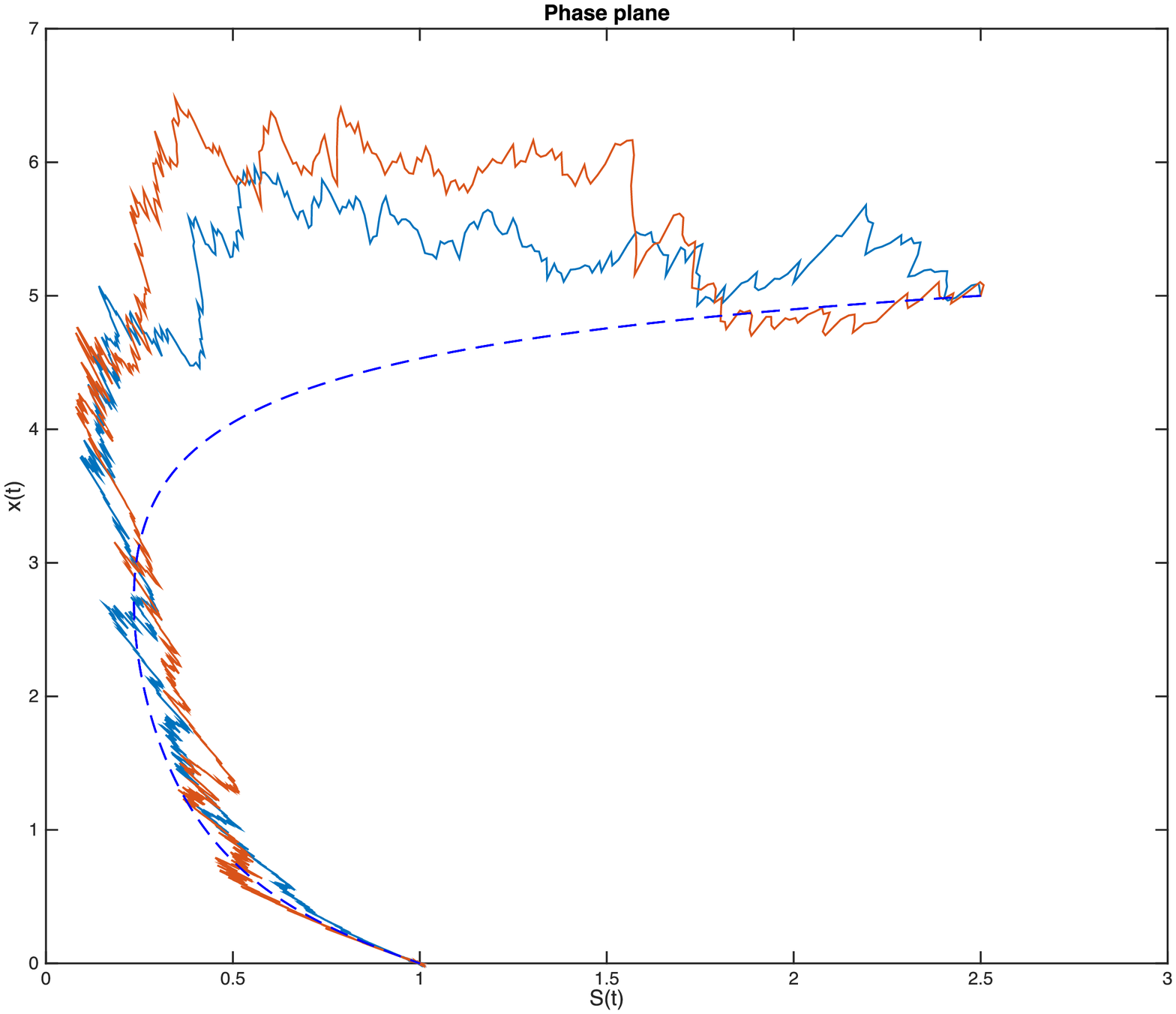}
\caption{Extintion. $\alpha=0.1$ (left) and $\alpha=0.5$ (right)}
\label{sim1}
\end{center}
\end{figure}

On the other hand, in Figure \ref{sim2} we take $D=3$ but, in this case, $\alpha=1$ (left) and $\alpha=1.5$ (right). Then, it follows that $\bar{D}=2$ (left) and $\bar{D}=2.6250$ (right) then, since $\mu(S^0)=1.8750$ and thanks to Proposition \ref{pe}, we also obtain the extinction of the species.

\begin{figure}[H]
\begin{center}
\noindent\includegraphics[scale=0.28]{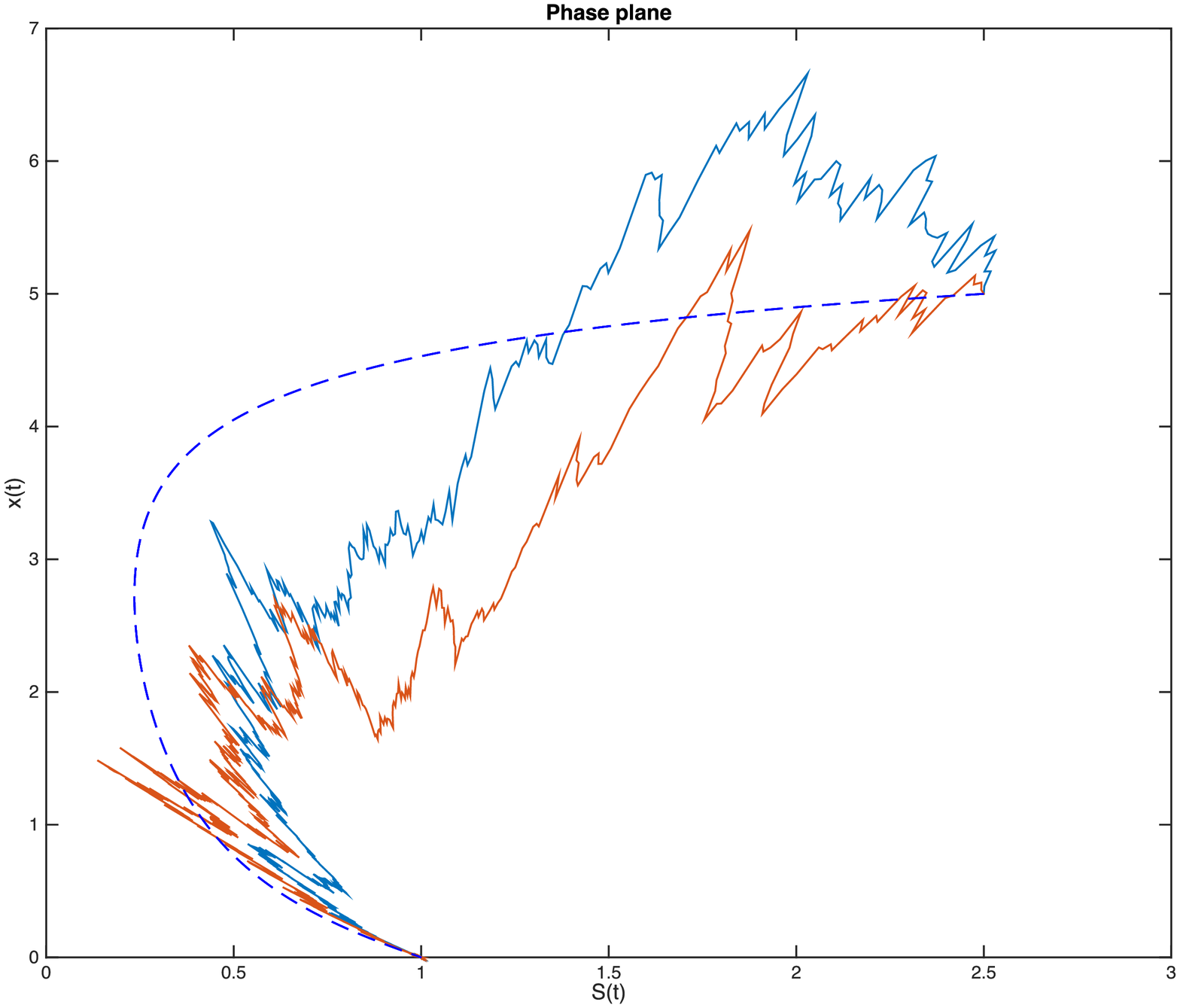}\hspace{-0.5cm}\includegraphics[scale=0.28]{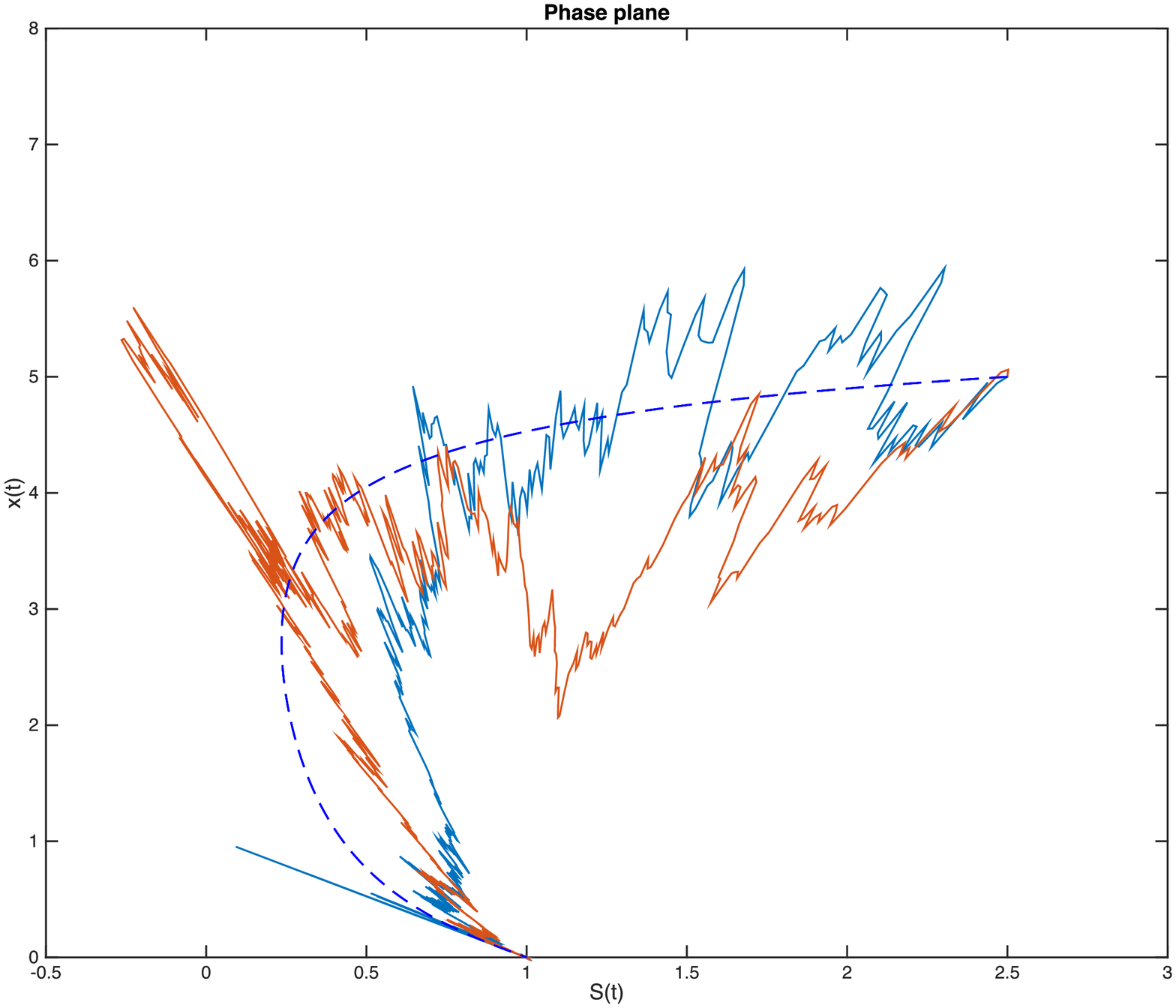}
\caption{Extinction. $\alpha=1$ (left) and $\alpha=1.5$ (right)}
\label{sim2}
\end{center}
\end{figure}

Now, in Figure \ref{sim3} we will take $D=1.5$ and we choose $\alpha=0.1$ (left) and $\alpha=0.5$ (right). Then, we can check that $\bar{D}=1.5050$ (left), $\bar{D}=1.6250$ (right) and $\mu(S^0)=1.8750$ thus, although it is not possible to ensure mathematically the persistence of the microbial biomass, we can get it for the previous values of the parameters, as we can see in the simulations.

\begin{figure}[H]
\begin{center}
\noindent\includegraphics[scale=0.28]{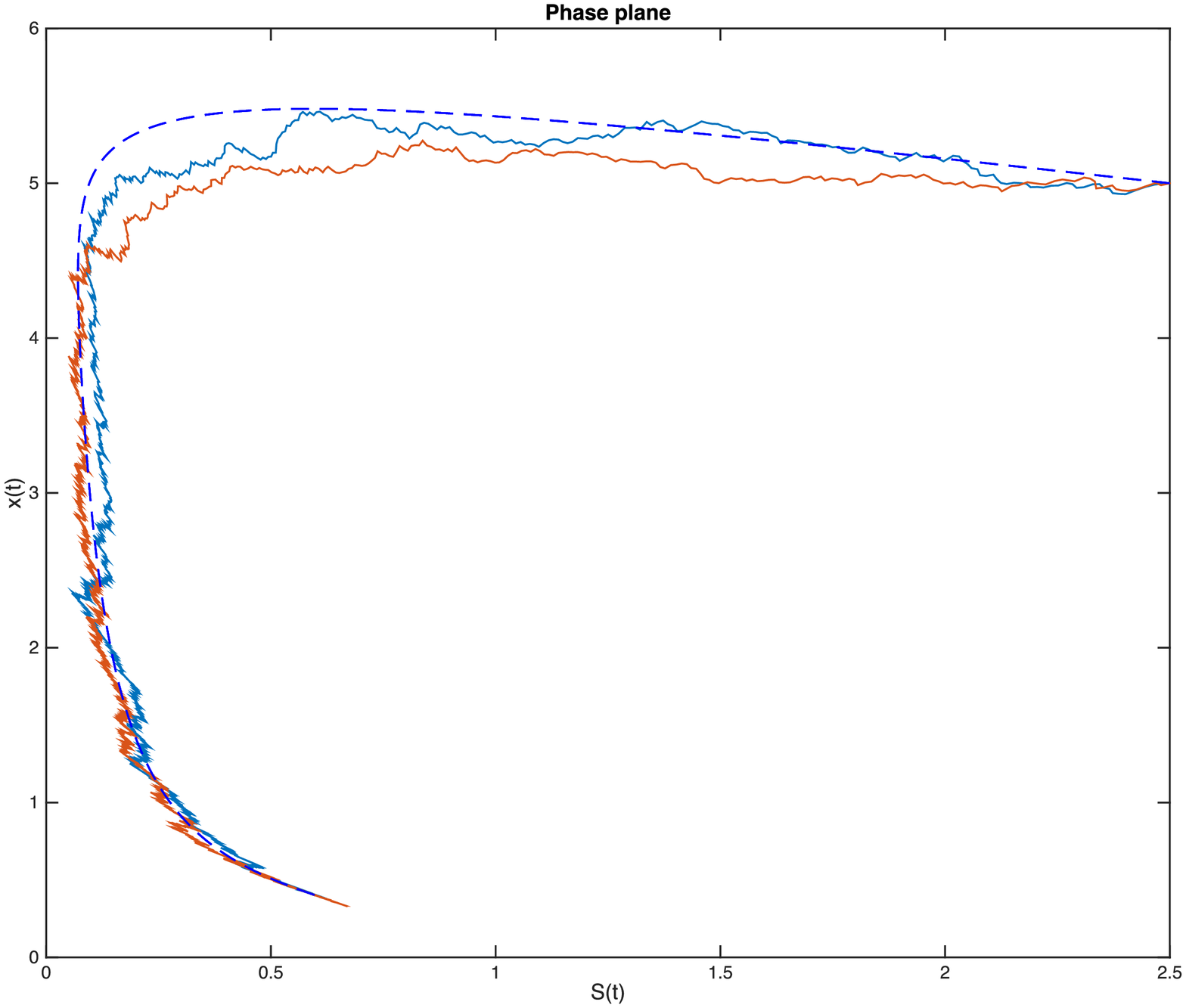}\hspace{-0.5cm}\includegraphics[scale=0.28]{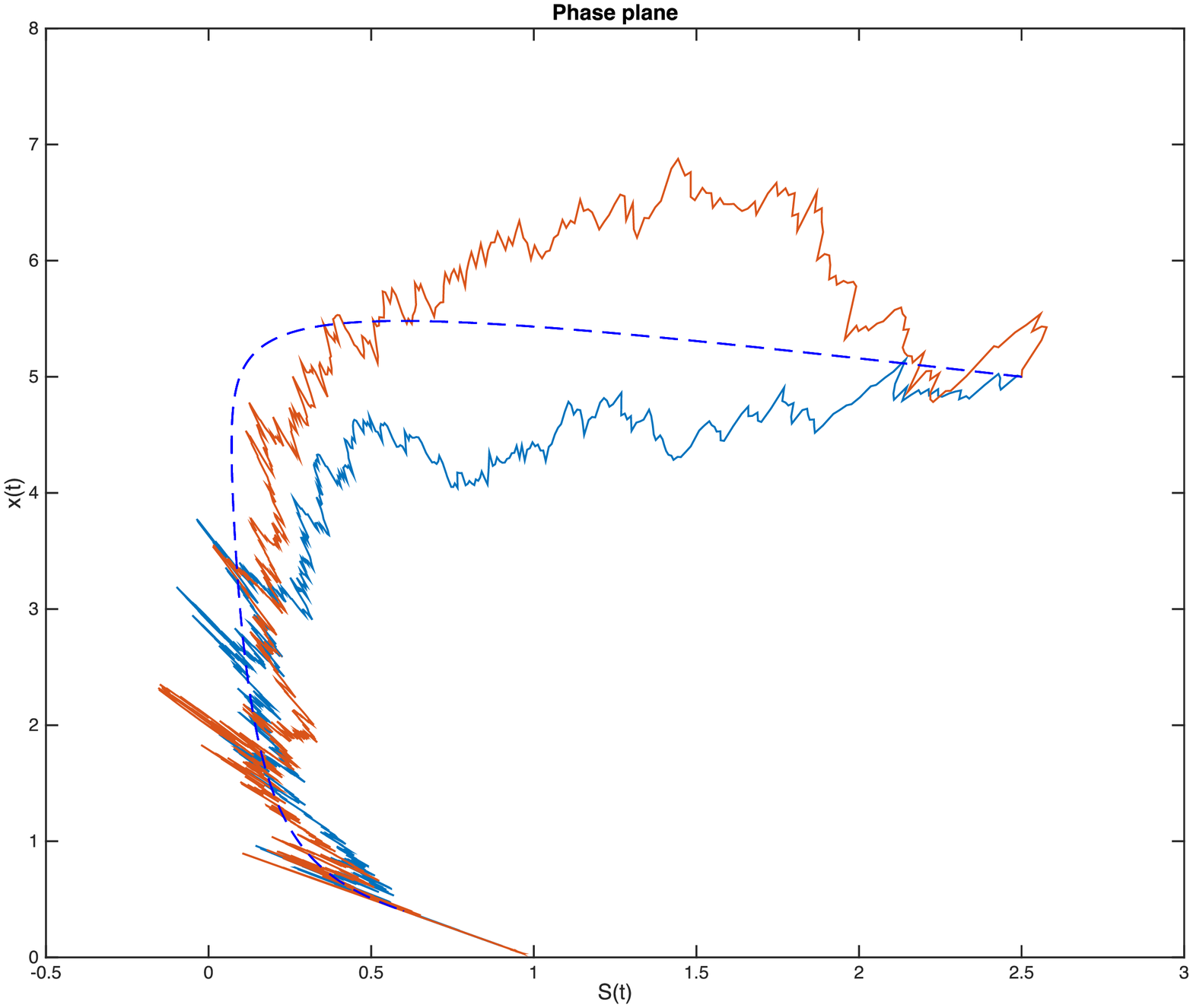}
\caption{Persistence. $\alpha=0.1$ (left) and $\alpha=0.5$ (right)}
\label{sim3}
\end{center}
\end{figure}

However, in Figure \ref{sim4} we take $D=1.5$, $\alpha=1$ (left) and $\alpha=1.5$ (right). Since condition \eqref{ce} holds true, it is not surprising to obtain the extinction of the microorganisms.

\begin{figure}[H]
\begin{center}
\noindent\includegraphics[scale=0.28]{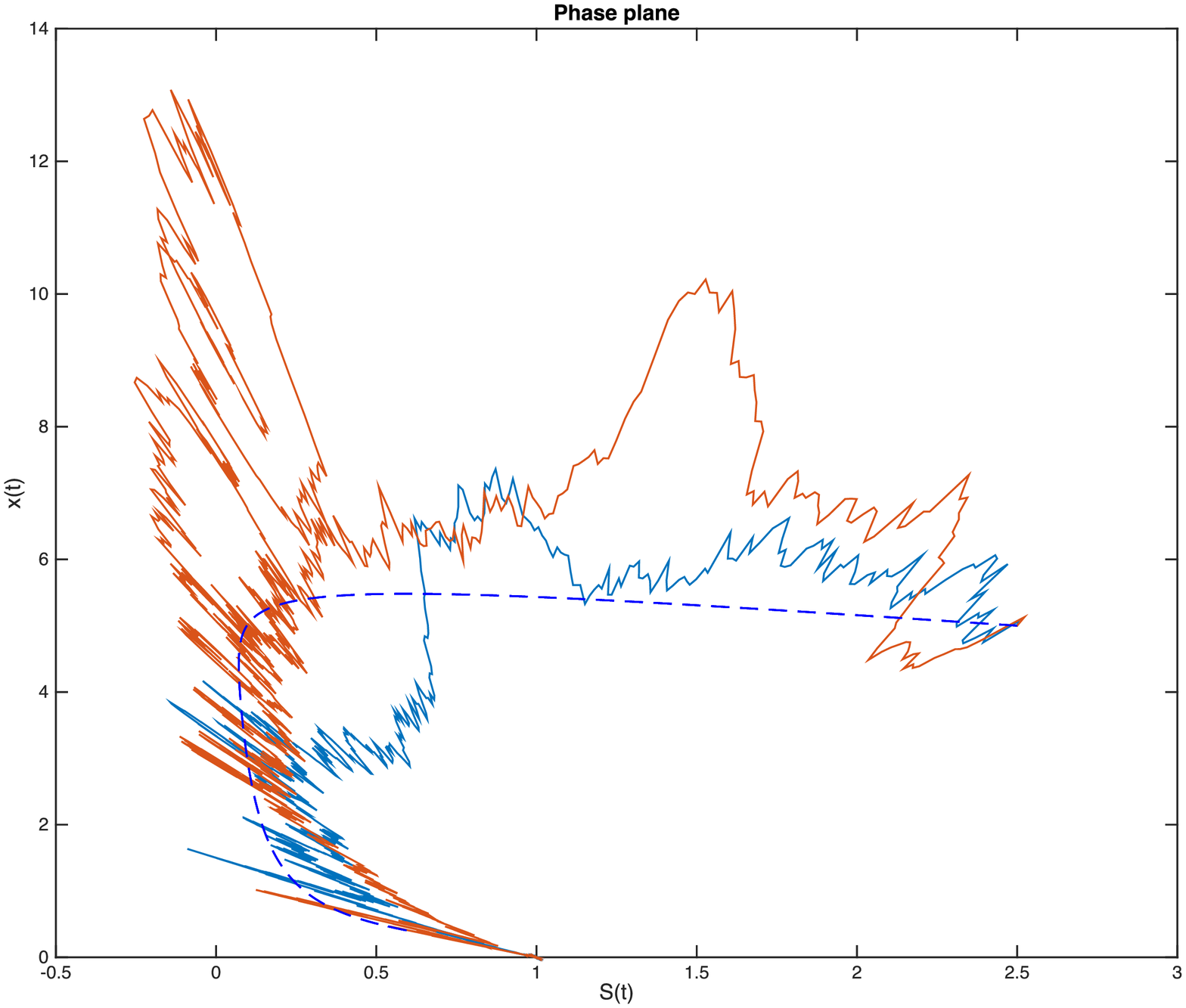}\hspace{-0.5cm}\includegraphics[scale=0.28]{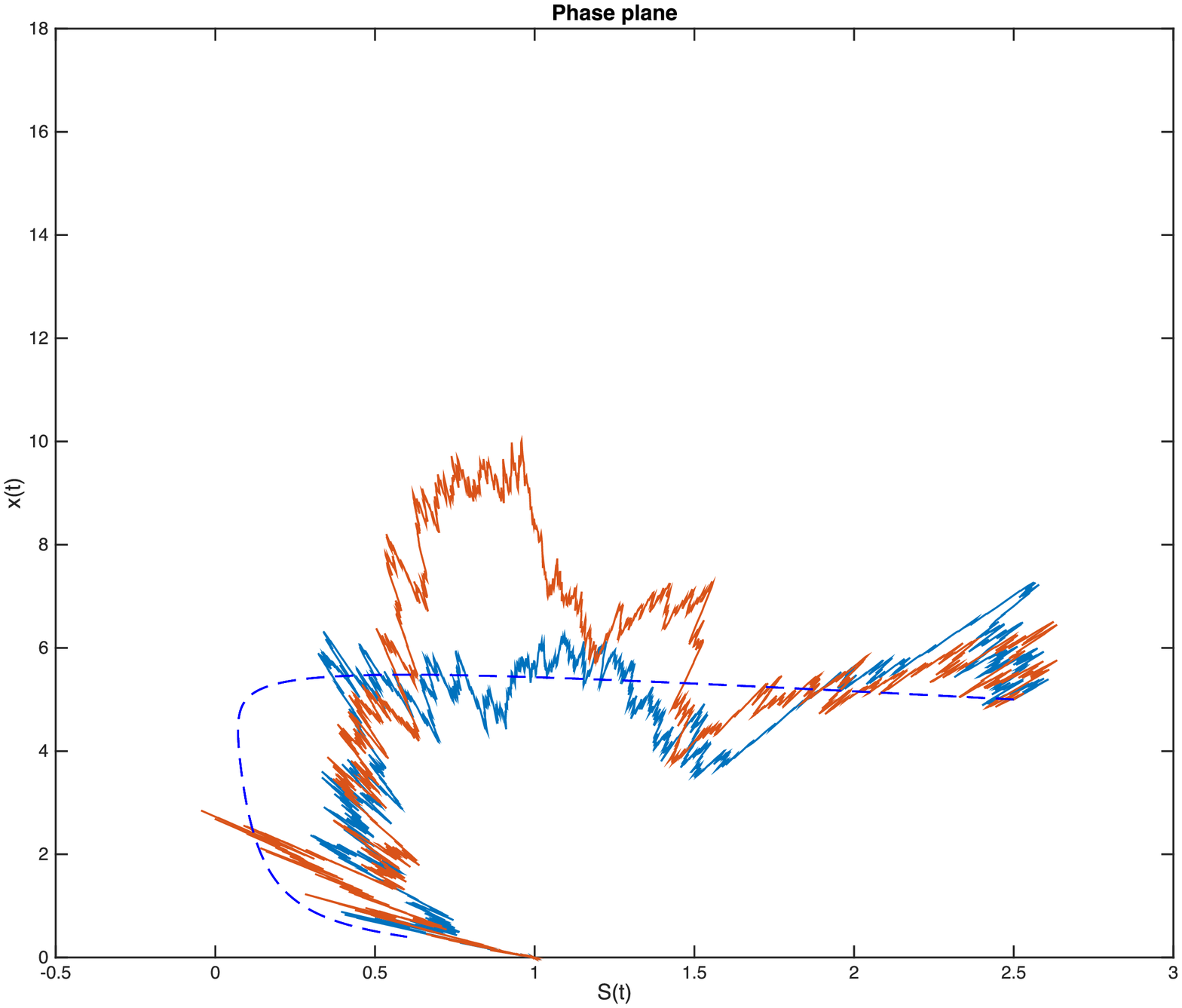}
\caption{Extinction. $\alpha=1$ (left) and $\alpha=1.5$ (right)}
\label{sim4}
\end{center}
\end{figure}

Finally, in Figure \ref{sim5} we will take $D=0.8$ and we will choose $\alpha=0.1$ (left) and $\alpha=0.5$ (right). It is easy to check that $\bar{D}=0.8050$ (left), $\bar{D}=0.9250$ (right) and $\mu(S^0)=1.8750$ thus, although it is not possible to guarantee mathematically the persistence of the species, since \eqref{ce} is not fulfilled, we can obtain it in this case.

\begin{figure}[H]
\begin{center}
\noindent\includegraphics[scale=0.28]{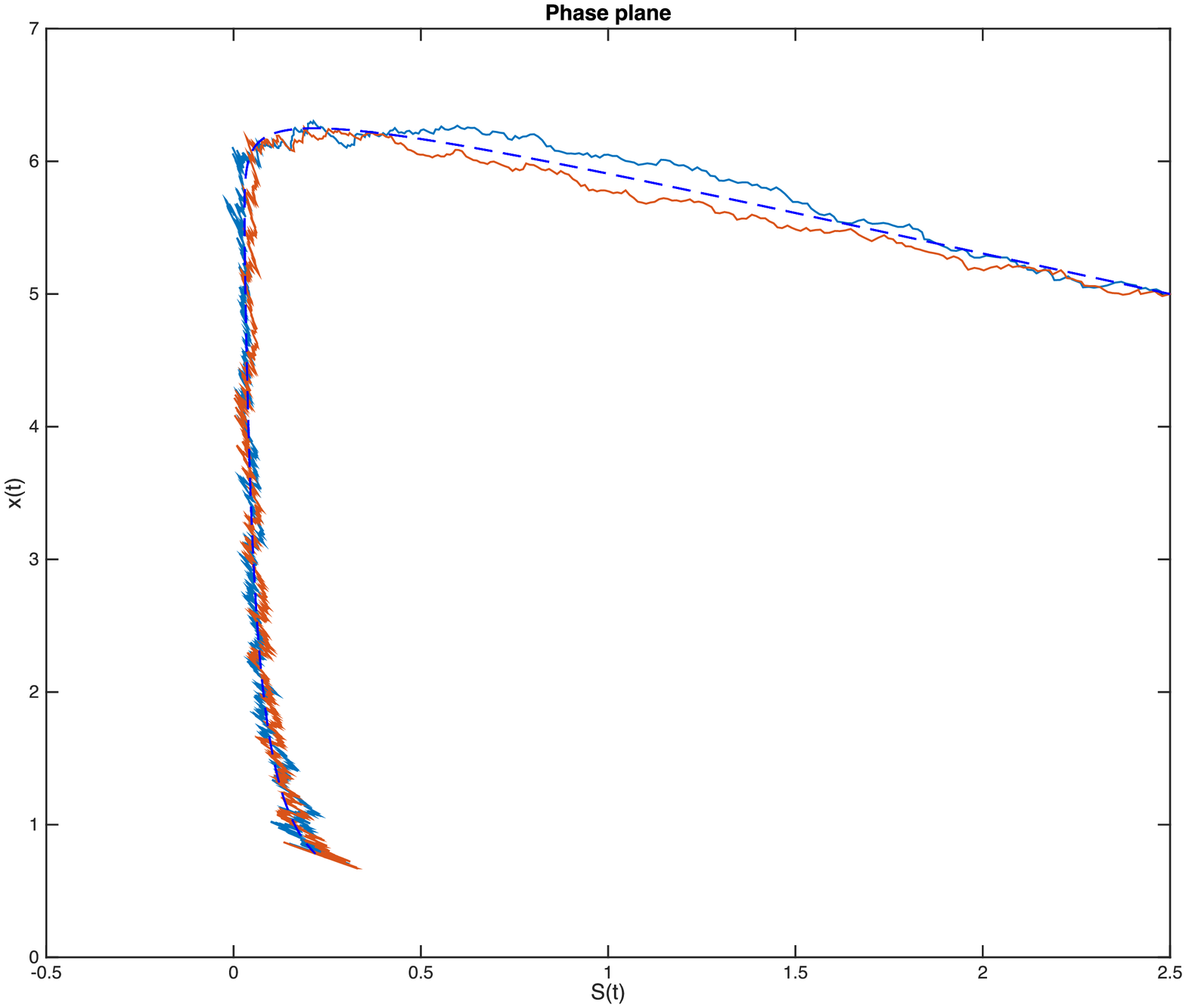}\hspace{-0.5cm}\includegraphics[scale=0.28]{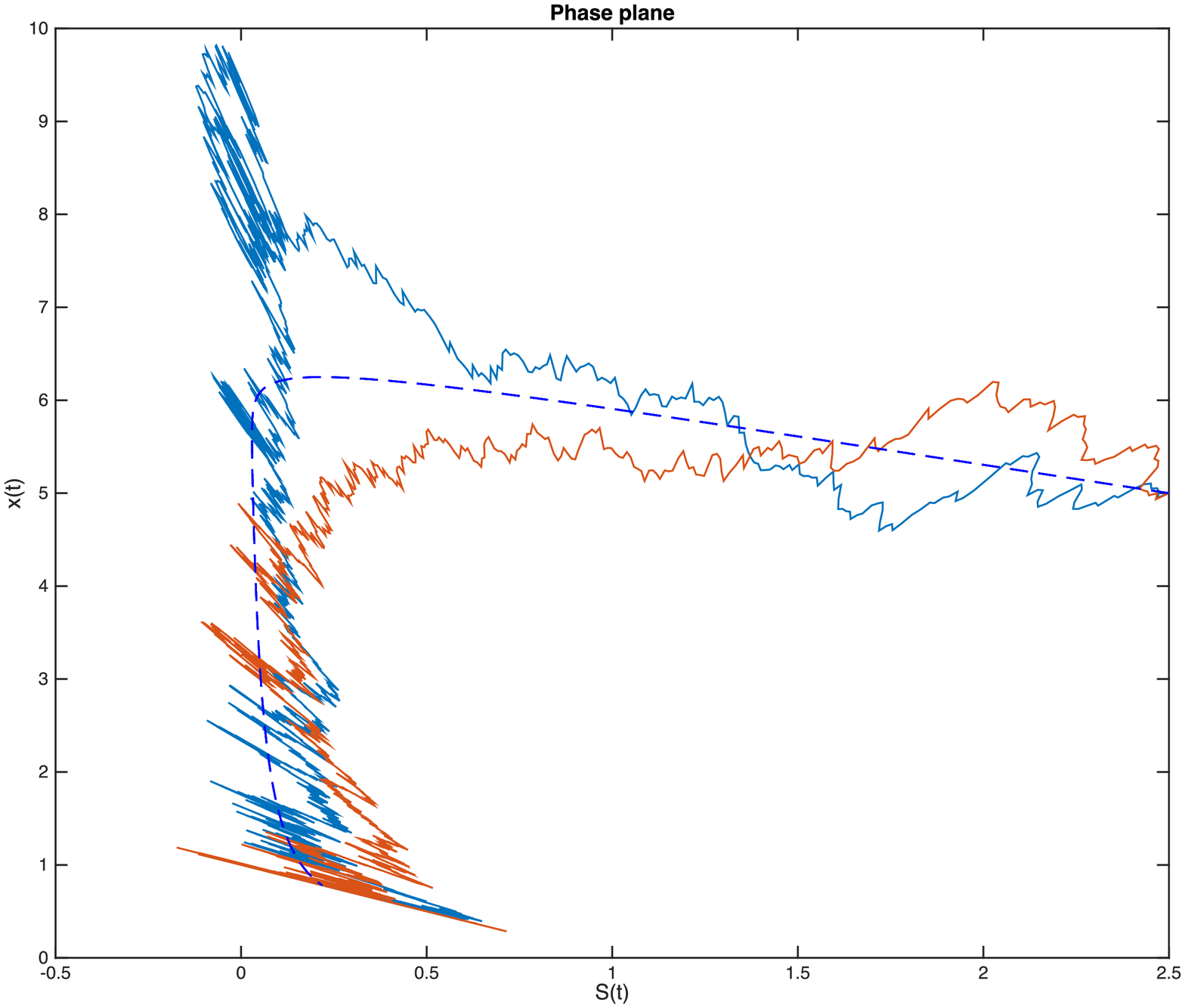}
\caption{Persistence. $\alpha=0.1$ (left) and $\alpha=0.5$ (right)}
\label{sim5}
\end{center}
\end{figure}

\begin{remark}
We would like to mention that the fact that the substrate $S$ (or its corresponding $\sigma$) may take negative values does not produce any mathematical inconsistence in our analysis, in other words, our mathematical analysis is accurate to handle the mathematical problem. However, from a biological point of view, this may reflect some troubles and suggests that either the fact of perturbing the dilution rate with an additive noise may not be a realistic situation, or that we should try to use a some kind of switching system to model our real chemostat in such a way that when the dilution may be negative we use a different equation to model the system. This will lead us to a different analysis in some subsequent papers by considering a different kind of randomness or stochasticity in this parameter or designing a different model for our problem.\n

On the other hand, it could also be considered a noisy term in each equation of the deterministic model in the same fashion as in the paper by Imhof and Walcher \cite{imhof}, which ensures the positivity of both the nutrient and biomass, although does not preserve the wash out equilibrium from the deterministic to the stochastic model (see e.g. \cite{CGLii} for more details about this situation). \end{remark}

\end{document}